\documentclass[a4paper,12pt]{article}
\usepackage{amsmath}
\usepackage{amsthm}
\usepackage{graphicx}
\usepackage{verbatim}
\usepackage{epsfig}
\usepackage{hyperref}
\usepackage{float}
\usepackage{color}
\usepackage{fullpage}
\usepackage{amsfonts}
\usepackage{amscd}
\usepackage{amssymb}
\usepackage{graphics}
\usepackage{amsmath}
\usepackage[english]{babel}
\usepackage{bbm}
\usepackage{yfonts}
\usepackage{mathrsfs}
\usepackage{color}
\DeclareFontFamily{OML}{rsfs}{\skewchar\font'177}
\DeclareFontShape{OML}{rsfs}{m}{n}{ <5> <6> rsfs5 <7> <8> <9>
rsfs7 <10> <10.95> <12> <14.4> <17.28> <20.74> <24.88> rsfs10 }{}
\DeclareMathAlphabet{\mathfs}{OML}{rsfs}{m}{n}

\newcommand{\prob}{{\bf P}}
\newcommand{\bae}{\begin{equation}\begin{aligned}}
\newcommand{\eae}{\end{aligned}\end{equation}}

\newcommand{\Z}{\mathbb{Z}}

\newtheorem{thm}{Theorem}
\newtheorem{prop}[thm]{Proposition}
\newtheorem{lem}[thm]{Lemma}

\newtheorem{conj}[thm]{Conjecture}

\def\proofof#1{{ \medbreak \noindent {\emph{Proof of #1.}} }}
\begin{document}
\title{Bernoulli and self-destructive percolation on non-amenable graphs}
\author{Daniel Ahlberg, Vladas Sidoravicius and Johan Tykesson\\ \normalsize\em Instituto Nacional de Matem\'atica Pura e Aplicada\thanks{Research in part carried out during the program \emph{Random Spatial Processes} at MSRI, Berkeley.}}
\maketitle
\begin{abstract}
In this note we study some properties of infinite percolation clusters on non-amenable graphs. In particular, we study the percolative properties of the complement of infinite percolation clusters. An approach based on mass-transport is adapted to show that for a large class of non-amenable graphs, the graph obtained by removing each site contained in an infinite percolation cluster has critical percolation threshold which can be arbitrarily close to the critical threshold for the original graph, almost surely, as $p\downarrow p_c$.
Closely related is the self-destructive percolation process, introduced by J.\ van den Berg and R.\ Brouwer, for which we prove that an infinite cluster emerges for any small reinforcement.
\end{abstract}


\section{Introduction}\label{s.intro}

In this note we consider Bernoulli site percolation and a variant thereof, self-destructive percolation, on transitive unimodular non-amenable graphs. In the self-destructive process, one starts with a supercritical Bernoulli configuration of open sites on an infinite transitive graph $G=(V,E)$, with parameter $p\in(0,1)$. Next all sites that belong to an infinite open cluster are declared closed. In the final step, the configuration, at this point consisting only of finite open clusters, is reinforced by an independent Bernoulli configuration with parameter $\delta>0$. Thus each site is given an extra chance to become open. The main question one asks is whether this last reinforcement step creates infinite open clusters.
By means of domination by product measures, as introduced by Liggett, Schonmann and Stacy~\cite{ligschsta97}, it is easily seen that for sufficiently small $p>p_c(G)$, the critical probability for percolation on $G$, there exists some $\delta$ strictly smaller than $p_c(G)$ such that an infinite open cluster reappears.
The challenge instead is to prove that arbitrarily small $\delta$ may be sufficient to create infinite open clusters for some $p>p_c(G)$. We will below prove that this is indeed the case if $G$ is non-amenable (see Theorem~\ref{t.mthm1}). At the same time we emphasize that van den Berg and Brouwer \cite{vdberg1} conjectured that self-destructive percolation on the square lattice does not possess this property.

Further properties of the self-destructive percolation process were earlier studied on planar lattices, such as the $\Z^2$ lattice and the binary tree ${\mathbb T}^2$, also in \cite{vdberg2} and \cite{vdberg3}.
Motivation to study the self-destructive process was found in an observed relation to forest fire models. However, working with the self-destructive process leads one to very detailed questions regarding infinite percolation clusters. In particular, one is led to understand how much damage the removal of the infinite clusters may cause, and first of all, if the removal causes any essential damage at all. In order to be a bit more precise, let $\omega_p$ denote a Bernoulli site percolation configuration at intensity $p$, and write $C^\infty(\omega_p)$ for the set of points in $\omega_p$ contained in an infinite cluster.
In the case of the usual cubic lattices we conjecture the following.

\begin{conj}
Assume $d\ge3$. For any $\delta>0$ there is $p_0>p_c(\Z^d)$ such that if $p_c(\Z^d)<p\le p_0$, then
$$
p_c(\Z^d)<p_c\big(\Z^d\setminus C^\infty(\omega_p)\big)< p_c(\Z^d)+\delta,\quad\text{almost surely}.
$$
\end{conj}

One of our main results states that the latter part of the above conjecture is true in the non-amenable setting (see Theorem~\ref{t.mthm2}). The results we present are based on an approach using mass-transport, and follows the lines of~\cite{BLPS}. Results analogous to the ones presented here have in parallel been obtained for the hyper-cubic lattice in sufficiently high dimensions, and are presented in \cite{ahldumkozsid13}.

\section{Notation and results}\label{s.notation}

Throughout the text, we let $G=(V,E)$ be an infinite, locally finite, connected graph with edge set $E$ and vertex set $V$. A \emph{graph automorphism} for $G$ is a bijective map $f:V\to V$ such that $f(u)$ and $f(v)$ are connected by an edge if and only if $u$ and $v$ are. Let Aut$(G)$ denote the group of graph automorphisms on $G$. $G$ is said to be \emph{transitive} if for any $u,v\in V$ there is a graph automorphism mapping $u$ to $v$. A transitive graph $G$ is said to be \emph{unimodular} if there exists a closed subgroup $H$ of Aut($G$) that acts transitively on $V$, such that for every $u,v\in V$
$$
|\textup{Stab}_H(u)v|=|\textup{Stab}_H(v)u|,
$$
where $\textup{Stab}_H(u)=\{h\in H:hu=u\}$.

Loosely speaking, a graph is transitive if it `looks the same' from each vertex, and is unimodular if for each pair $u$ and $v$, the number of vertices that `looks the same' as $u$ seen from $v$ equals the number of vertices `looking the same' as $v$ seen from $u$. In particular, Cayley graphs are transitive and unimodular. We will in this paper only consider transitive and unimodular graphs.

Finally, we define the concept of amenability. Known as the (vertex-)isoperimetric constant, let
$$
\kappa_V(G):=\inf_{W}\frac{|\partial_VW|}{|W|},
$$
where the infimum is taken over all finite subsets of $V$ and the boundary $\partial_VW$ is the set of sites in $W$ which has a neighbour in $V\setminus W$. $G$ is said to be \emph{amenable} if $\kappa_V(G)=0$, and \emph{non-amenable} otherwise.

We are concerned with \emph{site percolation} on $G$, which refers to a probability measure on $\{0,1\}^V$. It is natural to identify site percolation configurations $\omega\in\{0,1\}^V$ and subgraphs of $G$ induced by subsets of $V$. Hence, a site percolation measure on $G$ is likewise a probability measure on subgraphs of $G$. The site percolation measure which includes each site independently with equal probability is known as Bernoulli percolation. For any $\omega\subset V$, let $C^\infty(\omega)\subset\omega$ denote the set of vertices belonging to an infinite cluster, and let $N(\omega)$ denote the number of infinite clusters in $\omega$. We recall next the commonly used notation of $p_c=p_c(G)$ for the critical probability for Bernoulli (site) percolation on $G$, and the threshold $p_u=p_u(G)\in[p_c,1]$ for the existence of a unique infinite cluster.

It is known that uniqueness is monotone for transitive graphs, that is, there is an almost surely unique infinite cluster for all $p\in(p_u,1]$; see~\cite{HP,schonmann99}. For percolation on the $\Z^d$ lattice it is known that there may be at most one infinite cluster almost surely~\cite{harris60,aizkesnew87}, so $p_c(\Z^d)=p_u(\Z^d)$, but there exists more exotic graphs for which $p_c<p_u$ (e.g.\ the binary tree, but see~\cite{grinew90} for an example where $p_c<p_u<1$). A distinguishing factor in this context is the notion of amenability. More precisely, Benjamini and Schramm~\cite{bensch96} have conjectured that for infinite connected transitive graphs, $p_c<p_u$ if and only if the graph is non-amenable. A further discussion of these and related results is given in~\cite{hagjon06}.

Assign independent uniformly distributed random variables on $[0,1]$ to the vertices of $G$. Declare a site \emph{$p$-open} if its value is at most $p$, and let $\omega_p$ denote the set of $p$-open vertices. Clearly, if $p\le q$, then a $p$-open site is also $q$-open, and $\omega_p$ is Bernoulli distributed with intensity $p$. In particular, this provides a simultaneous coupling of Bernoulli site percolation for all $p\in[0,1]$. Let $(\omega_p)_{p\in[0,1]}$ and $(\tilde{\omega}_p)_{p\in[0,1]}$ be two independent Bernoulli configurations obtained in this way. The dependent site percolation process known as \emph{self-destructive percolation} with parameters $p$ and $\delta$ in $[0,1]$ is on $G$ defined as
\begin{equation*}
\Phi(p,\delta):=\big(\omega_p\setminus C^{\infty}(\omega_p)\big)\cup \tilde{\omega}_{\delta}.
\end{equation*}
(A more precise notation would be $\Phi(\omega_p,\tilde{\omega}_{\delta})$, but we prefer $\Phi(p,\delta)$ for simplicity.) In words, a configuration is obtained from a $p$-Bernoulli configuration from which all infinite clusters have been removed, but then reinforced by an independent Bernoulli configuration of intensity $\delta$. Since $\omega_p$ may contain infinite clusters only for $p\ge p_c$, the model is reduced to regular Bernoulli percolation for lower values of $p$. Due to ergodicity it is clear that the probability of $\Phi(p,\delta)$ containing an infinite cluster is either 0 or 1.
For fixed $p$, define the \emph{critical reinforcement threshold} $\delta_c(p)$ as
\begin{equation*}
\delta_c(p)=\delta_c(p,G):=\inf\{\delta\ge0:N(\Phi(p,\delta))\ge 1\mbox{ almost surely}\}.
\end{equation*}

 In~\cite{vdberg1} it was conjectured that $\delta_c(p)\ge c>0$ for all $p>p_c$ on the square lattice, and an extensive discussion in support of this conjecture was presented. In contrast, they proved that $\delta_c(p)\to0$ as $p\searrow p_c$ for the binary tree. Our first result extends the latter result to a wider class of non-amenable transitive graphs.

 \begin{thm}\label{t.mthm1}
 Suppose $G$ is a transitive unimodular non-amenable graph. Then, for every $\delta>0$ there exists $p>p_c$ such that $\Phi(p,\delta)$ almost surely percolates, i.e.,
 $$\lim_{p\searrow p_c}\delta_c(p)=0.
 $$
 \end{thm}

Theorem~\ref{t.mthm1} will be deduced from Proposition~\ref{p.helpprop1} below. Complementary to Theorem~\ref{t.mthm1}, there is more to be said about the structure of $\omega_{p+\delta}\setminus C^\infty(\omega_p)$ and the percolative properties of the graph $G\setminus C^\infty(\omega_p)$. The following will be derived also from Proposition~\ref{p.helpprop1}.

\begin{thm}\label{t.mthm2}
Suppose $G$ is a transitive unimodular non-amenable graph. For every $\delta>0$, there is $p>p_c(G)$ such that
$$
p_c\big(G\setminus C^\infty(\omega_p)\big)\,<\, p_c(G)+\delta,\quad\text{almost surely}.
$$
\end{thm}

A more precise version of the uniqueness monotonicity property (mentioned above) states that all infinite clusters are born simultaneously, in the sense that for any $p_2>p_1>p_c$, every infinite cluster in $\omega_{p_2}$ contains an infinite cluster of $\omega_{p_1}$, almost surely. This was obtained in increasing generality by \cite{alexander95}, \cite{HP} and \cite{schonmann99}. We contrast their result by proving that in the non-amenable setting, infinite clusters \emph{are} born in the supercritical regime in another sense; namely that $\omega_{p_2}\setminus C^\infty(\omega_{p_1})$ may contain infinite clusters.

\begin{thm}\label{t.mthm3}
Suppose $G$ is a transitive unimodular non-amenable graph. For every $\delta>0$, there exists $p>p_c$ such that $\omega_{p_c+\delta}\setminus C^\infty(\omega_p)$ contains an infinite cluster, almost surely.
\end{thm}

Statements analogous to Theorems~\ref{t.mthm1},~\ref{t.mthm2}, and~\ref{t.mthm3} hold also for bond percolation.

\section{Proofs}

In this section, we give proofs to the above stated theorems. They will all be based on the following proposition, whose proof will follow closely the proof in \cite{BLPS} showing that on transitive unimodular non-amenable graphs, there is no infinite cluster at criticality. Call a probability measure $\mu$ on subsets of $V$ \emph{Aut($G$)-invariant} if $\omega\subset V$ being distributed according to $\mu$ implies that $g(\omega)$ has distribution $\mu$, for each $g\in\textup{Aut}(G)$.

\begin{prop}\label{p.helpprop1}
Suppose $G$ is a transitive unimodular non-amenable graph and that $p>p_c$. Suppose that $(\Gamma_n)_{n\ge1}$ is a sequence of Aut$(G)$-invariant site percolations on $G$, with the property that $\lim_{n\to\infty} \prob[o\in \Gamma_n]=0$. Then there is some $N<\infty$ (depending on $p$ and on the law of $(\Gamma_n)_{n\ge1}$) such that
\begin{equation*}
\prob[N(\omega_p\setminus \Gamma_n)\ge 1]>0\mbox{ for }n\ge N.
\end{equation*}
\end{prop}

Observe that if $(\Gamma_n)_{n\ge1}$ fulfils the requirements of Proposition~\ref{p.helpprop1}, then the invariance property and a union bound implies that for any finite $K\subset V$,
\begin{equation}\label{e.invcons}
\lim_{n\to\infty}\prob[\Gamma_n\cap K\neq\emptyset]=0.
\end{equation}

Theorems~\ref{t.mthm1},~\ref{t.mthm2} and~\ref{t.mthm3} are easily and analogously deduced from Proposition~\ref{p.helpprop1}, based on the fact from \cite{BLPS} that transitive unimodular non-amenable graphs have no infinite cluster at criticality. However, note also that for all $p\ge p_c$ and $\delta>0$
\begin{equation*}
(\omega_{p_c}\cup\tilde{\omega}_{\delta})\setminus C^{\infty}(\omega_p)\subset\Phi(p,\delta).
\end{equation*}
Thus, Theorem~\ref{t.mthm1} follows readily from Theorem~\ref{t.mthm3}. Also Theorem~\ref{t.mthm2} is easily derived from Theorem~\ref{t.mthm3}, for which reason we will only present a proof of the latter.

\begin{proof}[Proof of Theorem~\ref{t.mthm3} from Proposition~\ref{p.helpprop1}]
It is well-known that $\theta(p)=\prob[o\in C^{\infty}(\omega_{p})]$ is right-continuous on $[0,1]$ (it is the limit of a decreasing sequence of increasing continuous functions). Since critical percolation on transitive unimodular non-amenable graphs produces no infinite clusters almost surely~\cite[Theorem~1.1]{BLPS}, it follows that
$$
\lim_{p\searrow p_c}\prob[o\in C^{\infty}(\omega_{p})]=0.
$$
Moreover, the law of $C^{\infty}(\omega_{p})$ is invariant under Aut$(G)$. Hence, since $\omega_{p_c+\delta}$ is supercritical, Proposition~\ref{p.helpprop1} shows that there is some $p_0>p_c$ such that $\omega_{p_c+\delta}\setminus C^\infty(\omega_{p_0})$ contains infinite clusters with positive probability. By ergodicity,  infinite clusters have to be contained in $\omega_{p_c+\delta}\setminus C^\infty(\omega_{p_0})$ with probability $1$.
\end{proof}


Although similar to the proof of Theorem~1.1 of \cite{BLPS}, we sketch the proof of Proposition~\ref{p.helpprop1} for completeness. The main tool therein is the mass-transport principle, and unimodularity is assumed for that purpose. The proof will be divided into two cases, depending on whether $p$ is in the regime of a unique or of multiple infinite clusters. We remark that in the case that $p_c<p_u$, it suffices to consider the case of infinitely many infinite clusters. Indeed, we recall that this was in~\cite{bensch96} conjectured to be the case for all transitive non-amenable graphs.

The proof for the regime of infinitely many infinite clusters will indirectly assume non-amenability in the following sense: There is at most one infinite cluster for Bernoulli percolation on a connected, transitive and amenable graph almost surely, due to the argument of Burton and Keane~\cite{burkea89}. In the regime of a unique infinite cluster, the proof will use the following characteristic feature of non-amenable graphs found in~\cite{BLPS}.

\begin{lem}\label{l.helplem1}
Let $G$ be a transitive unimodular non-amenable graph. There exists $\gamma=\gamma(G)>0$ such that, for every Aut$(G)$-invariant site percolation $\Gamma$ on $G$,
$$
\prob[v\not\in\Gamma]<\gamma\quad\text{implies}\quad
\prob[N(\Gamma)\ge1]>0.
$$
\end{lem}


\proofof{Proposition~\ref{p.helpprop1}: the regime of a unique infinite cluster}
Assume that $\omega_p$ contains a unique infinite cluster $U$. We will define a site percolation configuration $\xi_{p,n}$ based on $\omega_p$ and $\Gamma_n$. Denote by $d_G(\cdot,\cdot)$ the graph distance in $G$, and given $v\in V$, let $D(v)$ denote the neighbours of $v$, and $U(v)$ denote the set of vertices $u\in U$ that minimise $d_G(v,u)$. Let $\xi_{p,n}$ consist of all $v\in V$ such that $d_G(v,U)\le n$, and $\bigcup_{u\in D(v)\cup\{v\}}U(u)$ is contained in a connected component of $\omega_p\setminus\Gamma_n$. For each fixed $\omega_p$, we have
\begin{equation*}
\lim_{n\to\infty}\prob[v\not\in\xi_{p,n}|\,\omega_p]\to0,
\end{equation*}
due to~\eqref{e.invcons} and the fact that each pair of vertices in $U$ are at finite distance. Hence, by the Bounded Convergence Theorem
\begin{equation*}
\lim_{n\to\infty}\prob[v\not\in\xi_{p,n}]\to0.
\end{equation*}
We conclude, via Lemma~\ref{l.helplem1}, that $\xi_{p,n}$ contains an infinite cluster with positive probability for all large enough $n$. However, from construction it follows that if $\xi_{p,n}$ contains an infinite cluster $\xi^\infty$, then $\bigcup_{v\in\xi^\infty}U(v)$ is contained in the same connected component of $\omega_p\setminus\Gamma_n$, which has to be infinite. Hence, also $N(\omega_p\setminus\Gamma_n)\ge1$ has positive probability to occur, for the same values of $n$.
\qed

\proofof{Proposition~\ref{p.helpprop1}: a sketch for the regime of infinitely many infinite clusters}
Suppose that $p$ is such that $\omega_p$ contains infinitely many infinite clusters almost surely. Say that $v\in V$ is an \emph{encounter point} for $\omega_p$ if  $v$ belongs to an infinite cluster of $\omega_p$ and there are at least three neighbours of $v$ that belong to disjoint infinite clusters of $\omega_p\setminus\{v\}$. Let $Y$ denote the set of encounter points for $\omega_p$. It is known that the almost sure existence of infinitely many infinite clusters implies that $Y$ is almost surely non-empty \cite{burkea89}. Moreover, a mass-transport argument shows (see Lemma 4.2 in \cite{BLPS}) that if $v\in Y$, then every infinite cluster in $\omega_p\setminus\{v\}$ containing some neighbour of $v$ contains at least one further encounter point for $\omega_p$, almost surely.

Construct a random forest $F=(V(F),E(F))$ as follows: Let $V(F)=Y$, and consider a family $\{x(v)\}_{v\in V}$ of i.i.d.\ random variables, uniformly distributed on $[0,1]$. Suppose that $v\in Y$, and for every infinite cluster $Z$ of $\omega_p\setminus\{v\}$ containing a neighbour of $v$, let $Q(Z,v)$ be the set of encounter points (for $\omega_p$) in $Z$ that minimise the distance in $\omega_p$ to $v$. Among the vertices in $Q(Z,v)$, let $u(Z,v)$ denote the vertex $u$ that minimises $x(u)$. Take $E(F)$ to be the set of all pairs $\{v,u(Z,v)\}$, where $v\in Y$ and $Z$ is an infinite cluster of $\omega_p\setminus\{v\}$ containing some neighbour of $v$. From the construction, it follows that the law of $F$ is Aut($G$)-invariant. In addition, according to Lemma~4.2 of \cite{BLPS}, the degree of each vertex is almost surely at least $3$.

In \cite{BLPS}, it is shown that $F$ is almost surely a forest. In other words, $F$ does not contain any cycles. We now construct a subgraph $F_n$ of $F$ for $n\ge1$. Let $V(F_n)=V(F)$, and for any $\{u,v\}\in E(F)$, let $\{u,v\}\in E(F_n)$ if and only if $u$ and $v$ belong to the same connected component of $\omega_p\setminus \Gamma_n$. 

We now proceed by contradiction. Suppose that for every $n\ge1$, we have that $\omega_p\setminus \Gamma_n$ contains only finite components, almost surely. For any $v\in V(G)$ let $K_n(v)$ denote the set of vertices in $V(F_n)$ that belong to the same connected component of $\omega_p\setminus \Gamma_n$ as $v$. By our assumption, $K_n(v)$ is almost surely finite for every $n\ge1$. Let $\partial K_n(v)$ denote the set of vertices in $F_n$ that have some $F$-neighbour outside of $K_n(v)$. Using the same mass-transport argument as in \cite{BLPS} (it is in this step that the fact that $F$ is a forest is used) it follows that
\begin{equation}\label{e.forestboundary}
\prob[v\in Y]\,\le\, 2\, \prob[v\in Y,v\in \partial K_n(v)].
\end{equation}
However, from~\eqref{e.invcons} and the Bounded Convergence Theorem it follows that the right-hand side tends to $0$ as $n\to\infty$. But the left-hand side of~\eqref{e.forestboundary} is independent of $n$ and moreover strictly positive. Hence~\eqref{e.forestboundary} cannot hold for large $n$, showing that the assumption that $\omega_p\setminus \Gamma_n$ contains only almost surely finite clusters for all $n$ is false.
\qed

\newcommand{\noopsort}[1]{}\def\cprime{$'$} \def\cprime{$'$}

\end{document}